\documentclass[12pt]{article}
\usepackage{amsmath,amsfonts,amssymb}
\usepackage{amsthm}
\def\diag{\mathrm{diag\,}}

\newtheorem{theorem}{Theorem}
\newtheorem{lemma}{Lemma}

\newtheorem{proposition}{Proposition}
\theoremstyle{remark}

\theoremstyle{definition}
\newtheorem{definition}{Definition}
\newtheorem{example}{Example}

\newcommand{\col}{\mbox{col\,}}

\newcommand{\R}{{R^1}}

\newcommand{\C}{{\mathbb C}}
\newcommand{\N}{{\mathbb N}}
\newcommand{\SSS}{{\mathbb S}}

\newcommand{\ran}{{\mbox{ran}}}

\date{}

\title{ One-dimensional Schr\"{o}dinger operators\\ with
  $\delta'$-interactions\\ on a set of Lebesgue measure zero }

\author{ Johannes F. Brasche\thanks{ Institute of Mathematics, TU
Clausthal, Clausthal--Zellerfeld, 38678, Germany, e-mail:
johannes.brasche@tu-clausthal.de } and Leonid Nizhnik\thanks
     {Institute of Mathematics, 3 Tereshchenkivs'ka, Kyiv,
     01601,
    Ukraine,  e-mail: nizhnik@imath.kiev.ua} }

\begin{document}
\maketitle

\begin{abstract}

  We give an abstract definition of a one-dimensional Schr\"odinger
  operator with $\delta'$-interaction on an arbitrary set~$\Gamma$ of
  Lebesgue measure zero. The number of negative eigenvalues of such an
  operator is at least as large as the number of those isolated points
  of the set~$\Gamma$ that have negative values of the intensity
  constants of the $\delta'$-interaction. In the case where the
  set~$\Gamma$ is endowed with a Radon measure, we give constructive
  examples of such operators having an infinite number of negative eigenvalues.
\end{abstract}
\

\section{Introduction}\label{S:1}

\sloppy One important problem in the theory of singular
perturbations of a Schr\"{o}dinger operator is to construct
non--trivial self--adjoint operators that describe interactions on
a set $\Gamma$ of Lebesgue measure zero [3,4]. The most studied
case is the one where $\Gamma$ consists of isolated points. In
this case the corresponding interaction is called point
interaction and leads to solvable models in quantum mechanics
[3,4].

\fussy For an arbitrary closed set~$\Gamma$ of Lebesgue measure
zero, the Schr\"odinger operator with interaction on~$\Gamma$ is
defined as a self-adjoint extension of the minimal
operator~$-\frac{d^{2}}{dx^{2}}$ defined on functions in the
space~$C_0^{\infty}(R^1 \setminus \Gamma)$~[3,4,8,24],  In some
cases, other definitions of the Schr\"odinger operator with
interaction on~$\Gamma$ are possible. Such definitions are given
in terms of certain boundary conditions~[3,4], singular
perturbations~[4,5], quadratic forms~[1,14], construction of
BVS~[22,23], and other methods~[30]. If~$\Gamma$ is endowed with a
Radon measure, then Schr\"odinger operators with interactions
on~$\Gamma$ can be defined using analogues of the usual boundary
conditions on~$\Gamma$~[8,24].

\fussy In this paper, we give an abstract definition of a
Schr\"odinger operator~$L_{\Gamma, \delta'}$
with~$\delta'$-interaction on an arbitrary set~$\Gamma$ of
Lebesgue measure zero. If the set~$\Gamma$ contains isolated
points, then functions from the domain of such an operator satisfy
the usual boundary conditions for the $\delta'$-interaction with
some intensities in the isolated points of~$\Gamma$. In this case,
the number of negative eigenvalues of the operator~$L_{\Gamma,
\delta'}$ is not less than the number of isolated points
of~$\Gamma$ having negative intensities of $\delta'$-interaction
(Theorem~1). If~$\Gamma$ is endowed with a Radon measure, then the
Schr\"odinger operator with $\delta'$-interaction on~$\Gamma$ can
also be defined using boundary conditions on~$\Gamma$ (Theorem~2).
We give constructive examples of Schr\"odinger operators with
$\delta'$-interactions on~$\Gamma$ having an infinite number of
negative eigenvalues (Theorem~3). The classification of point
interactions for a one--dimensional Schr\"odinger
  operator is briefly  given in section 2.
In section 8 we give the
deficiency subspaces of the minimal operator so that it becomes possible
to determine the set of all Schr{\"o}dinger operators describing an interaction
which takes place inside $\Gamma$.

\section{Point interactions}\label{S:2}

The one-dimensional Schr\"odinger operator that describes a
one-point inter\-action in a point~$x_0$ is a self-adjoint
operator on the space~$L_2(R^{1})$ and, for~$x \neq x_0$, is given
by the differential expression~$-\frac{d^2}{dx^2}$. The maximal
domain of the operator~$-\frac{d^2}{dx^2}$ for~$x \neq x_0$ is the
Sobolev space~$W_2^2(R^1 \setminus \{x_0\})$. For
functions~$\varphi,\,\psi \in W_2^2(R^1\setminus \{x_0\})$, we
have the Lagrange formula
\begin{equation}\label{eq:1}
  (- \psi'', \varphi)_{L_2}-( \psi,- \varphi'')_{L_2}= \omega(\Gamma
  \psi,\Gamma \varphi),
\end{equation}
where the boundary form~$\omega$ is defined on the space~$E^4$ of
boundary values of the functions~$\psi$ and~$\varphi$,
$$
 \Gamma \psi= \col
(\psi(x_0+0),\psi(x_0-0), \psi'(x_0+0),\psi'(x_0-0)) \in E^4,
$$
by the formula
\begin{multline}\label{eq:2}
  \omega (\Gamma\psi, \Gamma\varphi)=\psi'(x_0+0)
  \bar{\varphi}(x_0+0)-
  \psi(x_0+0)\bar{\varphi'}(x_0+0)
  \\
  -\psi'(x_0-0)
  \bar{\varphi}(x_0-0)+ \psi(x_0-0)\bar{\varphi'}(x_0-0).
\end{multline}
Self-adjoint restrictions of the maximal operator are defined by
domains in terms of the corresponding boundary data that make a
Lagrangian plane in the space~$E^4$; it is a maximal subspace on
which the boundary form satisfies~$\varpi(\Gamma \psi,\Gamma
\psi)=0$. Since the boundary form~(\ref{eq:2}) can be represented
as
\begin{equation}\label{eq:3}
  \omega(\Gamma \psi,\Gamma \varphi)=  (\Gamma_1 \psi,\Gamma_2
  \varphi)_{E^2}-(\Gamma_2 \psi,\Gamma_1 \varphi)_{E^2},
\end{equation}
where~$\Gamma_1 \psi=\col (\psi'(x_0+0),-\psi'(x_0-0)),\, \Gamma_2
\psi= \col (\psi(x_0+0),\psi(x_0-0))$, the general self-adjoint
boundary conditions are given by a unitary matrix~$U$ operating on
the space~$E^2$,
\begin{equation}\label{eq:4}
\Gamma_1 \psi+i\Gamma_2 \psi=U(\Gamma_1 \psi-i\Gamma_2\psi).
\end{equation}
The matrix~$U$ uniquely parametrizes the Lagrangian planes. This
gives rise to a Schr\"odinger operator~$A_U$ on the
space~$L_2(R^{1})$ with the  domain consisting of all functions in
the space~$W_2^2(R^1 \setminus \{x_0\})$ satisfying boundary
condition~(4) and~$ A_U \psi= -\psi''(x),\,\,\,x\neq x_0 $. The
Schr\"odinger operator~$A_U$ that describes a point interaction in
the point~$x_0$ is characterized with the matrix~$U$.

Conditions~(4) contain split boundary conditions of the form
\begin{equation}\label{eq:5}
  \begin{array}{l}
    \psi(x_0+0)\cos \alpha_+- \psi'(x_0+0)\sin \alpha_+=0,\\[2mm]
    \psi(x_0-0)\cos \alpha_- - \psi'(x_0-0)\sin \alpha_-=0,
  \end{array}
\end{equation}
where~$\alpha_{\pm}\in (-\frac{\pi}{2},\frac{\pi}{2}]$. These
boundary conditions define a non-transparent interaction in the
point~$x_0$. Conditions~(5) correspond to a self-adjoint
Schr\"odinger operator~$A$ on the space~$L_2(R^{1} )=L_2(-\infty,
x_0 )\oplus L_2(x_0,+\infty)$. This operator can be decomposed
into the direct sum~$A=A_1 \oplus A_2$ of self-adjoint
operators~$A_1$ and~$A_2$ acting on the spaces~$L_2(-\infty, x_0
)$ and~$L_2(x_0,+\infty)$ that correspond to boundary
conditions~(5) in the points~$x=x_0-0$ and~$x=x_0+0$,
respectively.

A converse statement also holds true. If a self-adjoint
Schr\"odinger operator~$A$ describes a one point interaction and
admits a representation as a direct sum,~$A=A_1 \oplus A_2$, then
the functions in its domain satisfy boundary conditions~(5) with
some real numbers~$\alpha_{\pm}$.

Boundary conditions~(4) split if and only if the unitary
matrix~$U$ is diagonal,~$U=\diag (e^{2i \alpha_+}, e^{- 2i
\alpha_-}) $. In this case, boundary conditions~(4) are equivalent
to conditions~(5).

The one-dimensional Schr\"odinger operator corresponding to point
inter\-actions on a finite set~$X=\{x_1,...,x_n\}$ is a
self-adjoint extension, to the space~$L_2(R^1)$, of the minimal
operator~$L_{min,
  X}$ defined on the space
$C_0^{\infty}(R^1\setminus X)$ by~$L_{min,
  X}\varphi(x)=-\varphi''(x)$ [3, 4]. All such self-adjoint extensions are
described by Lagrangian planes in the Euclidean space~$E^{4n}$ of
boundary data for the functions~$\psi \in W_2^2(R^1\setminus X)$.
This leads to self-adjoint boundary
conditions given by unitary matrices acting on~$E^{2n}$. Localized
self-adjoint boundary conditions have the form of~(4) in every
point~$x_k\in X$, whereas localized indecomposable boundary
conditions have the form [3]
\begin{equation}\label{eq:6}
  \col( \psi(x_k+0), \psi'(x_k+0))=\Lambda_k\col( \psi(x_k-0), \psi'(x_k-0)),
\end{equation}
where the transmission matrices $\Lambda_k$ can be written as
$\Lambda_k=e^{i \eta_k}R_k$, where~$R_k$ is a real matrix, and
$\det R_k=1,$\, $\eta_k$ is a real constant.

The boundary form~(2) can be represented equivalently as
\begin{equation}\label{eq:7}
  \omega(\Gamma\psi, \Gamma\varphi)=(\hat{\Gamma}_1\psi,
  \hat{\Gamma}_2 \varphi)_{E_2}-(\hat{\Gamma}_2\psi, \hat{\Gamma}_2
  \varphi)_{E_2},
\end{equation}
where
\begin{equation}\label{eq:8}
  \hat{\Gamma}_1\psi= \col(\psi'_s, \psi_s),\,\,\, \hat{\Gamma}_2
  \psi=\col(\psi_r, -\psi'_r),
\end{equation}
\begin{equation}\label{eq:13}
  \begin{array}{l}
    \psi_s=\psi(x_0+0) - \psi(x_0-0); \qquad  \psi'_s=
    \psi'(x_0+0)-\psi'(x_0-0);\\ \psi_r= \frac{1}{2}[\psi(x_0+0)+
    \psi(x_0-0)];\quad \psi'_r=\frac{1}{2}[\psi'(x_0+0)+
    \psi'(x_0-0)].
  \end{array}
\end{equation}

By~(\ref{eq:7}), general self-adjoint boundary conditions in the
point~$x_0$ are defined with a unitary matrix~$\hat{U}$ acting on
the space~$E^2$ and have the form
\begin{equation}\label{eq:9}
  \hat{\Gamma}_1\psi+i\Gamma_2\psi=
  \hat{U}(\hat{\Gamma}_1\psi-i\Gamma_2\psi).
\end{equation}
The matrices~$\hat{U}$ and~$U$ in the boundary conditions~(4)
and~(\ref{eq:9}) are connected with each other via the relations
$$
\begin{array}{l}
\hat{U}=(3CUC+1)(3+3CUC)^{-1},\\
U=C^*(3-\hat{U})(3\hat{U}-1)C,
\end{array}
$$
where~$C=\frac{1}{\sqrt{2}} \left(\begin{array}{c r} 1 & -i \\
1  &   i
\end{array}
\right)$ is a unitary matrix.

Among one-point interactions, the following four cases are
important.
\begin{itemize}
\item [1)] The~$\delta$-interaction, or $\delta$-potential, with
  intensity~$\alpha$ is defined by the boundary conditions
  \begin{equation}\label{eq:10}
    \psi(x_0+0)-\psi(x_0-0)=0,\qquad \psi'(x_0+0)-\psi'(x_0-0)=\alpha
    \psi_r(x_0),
  \end{equation}
  where~$x_0$ is the interaction point. In this case, the
  $\Lambda$-matrix in the boundary conditions~(\ref{eq:6})
  has the form~$\Lambda= \left(\begin{array}{c r} 1 & 0 \\
      \alpha & 1
    \end{array}
  \right)$.
\item [2)] The $\delta'$-interaction with intensity~$\beta$ is defined
  by the boundary conditions
  \begin{equation}\label{eq:11}
    \psi'(x_0+0)-\psi'(x_0-0)=0,\qquad \psi(x_0+0)-\psi(x_0-0)=\beta
    \psi'_r(x_0).
  \end{equation}
  In this case, the $\Lambda$-matrix in the boundary
  conditions~(\ref{eq:6}) has the
  form~$\Lambda= \left(\begin{array}{c r} 1 & \beta \\
      0 & 1
    \end{array}
  \right)$.
\item [3)] The $\delta'$-potential with intensity~$\gamma$ is defined
  with the boundary conditions
  \begin{equation}\label{eq:12}
    \psi(x_0+0)-\psi(x_0-0)=\gamma \psi_r(x_0) ,\qquad
    \psi'(x_0+0)-\psi'(x_0-0)=-\gamma \psi'_r(x_0).
  \end{equation}
  An equivalent form of the boundary conditions~(\ref{eq:12})
  is~$\psi(x_0+0)=\theta\psi(x_0-0)$,
  $\psi'(x_0+0)=\theta^{-1}\psi'(x_0-0)$,
  where~$\displaystyle \theta=\frac{2+\gamma}{2-\gamma}$. In
  this case, the matrix~$\Lambda$ in the boundary conditions~(\ref{eq:6})
  is~$\Lambda= \left(\begin{array}{c l} \theta & 0 \\
      0 & \theta^{-1}
    \end{array}
  \right)$.
\item [4)] The $\delta$-magnetic potential with intensity~$\mu$ is
  defined in terms of the boundary conditions
  \begin{equation}\label{eq:13a}
    \psi(x_0+0)-\psi(x_0-0)=i \mu \psi_r(x_0) ,\quad
    \psi'(x_0+0)-\psi'(x_0-0)=i \mu \psi'_r(x_0),
  \end{equation}
  where~$i$ is the imaginary unit. An equivalent form of the boundary
  conditions~(\ref{eq:13a}) is~$\psi(x_0+0)=e^{i\eta}\psi(x_0-0)$,
  $\psi'(x_0+0)=e^{i\eta}\psi'(x_0-0)$, where~$\frac{\mu}{2}=
  \tan\frac{\eta}{2}$. In this case, $\Lambda$ in the boundary
  conditions~(\ref{eq:6}) is a multiple of the identity
  matrix,~$\Lambda=e^{i\eta}I$.
\end{itemize}

To explain the names and the physical meaning of the four types of
interactions listed above, consider at first the formal
Schr\"odinger operators~$L$,
\begin{equation}\label{eq:14a}
  L=-\frac{d^2}{dx^2}+\varepsilon \delta^{(j)}(x-x_0),\,\,
  j=0,1;\,\,
  \varepsilon=\alpha,\,\,j=0;\,\,\varepsilon=\gamma,\,\,j=1,
\end{equation}
the expression~$L\psi$ can be defined in the sense of distribution
theory for functions~$\psi \in W^{2}_2(R^1\setminus \{x_0 \})$.

Indeed, the expression~$-\frac{d^2}{dx^2}$ on such functions~$
\psi$, in the sense of distribution theory, is given by the
expression
\begin{equation}\label{eq:15}
  -\frac{d^2}{dx^2}\psi(x)=-\psi''(x)-\delta'(x-x_0)\psi_s(x_0)-\delta(x-x_0)\psi'_s(x_0).
\end{equation}

The product~$\delta^{(j)}(x-x_0)\psi(x) $ is well defined if~$
\psi \in C^{\infty}(R^1)$, that is, the function~$\psi$ is a
multiplicator for the Schwartz space~$C_0^{\infty}(R^1)$ of test
functions. In this case,
\begin{equation}\label{eq:16}
  \delta(x-x_0)\psi(x)=\psi(x_0)\delta(x-x_0),\,\,\delta'(x-x_0)\psi(x)=\psi_r(x_0)\delta'(x-x_0)-\psi_r'(x_0)\delta(x-x_0).
\end{equation}
The identity~(\ref{eq:16}) can be extended as to also encompass
discontinuous functions~$\psi \in C^{\infty}(R^1\setminus \{x_0
\})$ by defining the functionals~$\delta^{(j)}(x-x_0)$
by~$(\delta^{(j)}(x-x_0),
\varphi(x))=(-1)^j\varphi_r^{(j)}(x_0)$~[4]. Hence, with such a
definition, formulas~(\ref{eq:16}) hold if all~$\psi^{(j)}(x_0)$
in the right-hand sides of formulas~(\ref{eq:16}) are replaced
with~$\psi_r^{(j)}(x_0)$.

If (\ref{eq:15}) and~(\ref{eq:16}) are used in~(\ref{eq:14a}),
then the condition~$L\psi \in L_2(R^1)$ leads to~(\ref{eq:10})
if~$j=0$ and to~(\ref{eq:12}) if~$j=1$.

Consider now a one-dimensional Schr\"odinger operator with
magnetic field potential~$a$ and potential~$V$, that is,~$L = (i
\frac{d}{dx}+a)^2+V$, in the particular case where~$L=
-\frac{d^2}{dx^2}+2ia \frac{d}{dx} +ia'$ and~$a(x)=\mu \delta(x)$,
so that
\begin{equation}\label{eq:18}
  L_{\mu} \psi= - \frac{d^2\psi}{dx^2}+2ia \delta(x)
  \frac{d\psi}{dx}+i\mu \delta'(x)\psi(x).
\end{equation}
If we use expressions~(\ref{eq:15}),~(\ref{eq:16})
in~(\ref{eq:18}), then imposing the condition on~$\psi(x) \in
W^{2}_2(R^1\setminus \{x_0 \})$ that the distribution~$L_{\mu}
\psi$ is a usual function in~$L_2(R^1)$ leads to~(\ref{eq:13a}).
Hence, the boundary conditions~(\ref{eq:13a}) describe a magnetic field
with the potential~$a(x)=\mu \delta(x)$.

Particular forms of the boundary conditions~(11)--(14) can be
represented as
\begin{equation}\label{eq:18a}
  \left(\begin{array}{c} \psi_s'(x_0)\\
      \psi_s(x_0)
    \end{array}
  \right)= B \left(\begin{array}{c} \psi_r(x_0)\\
      -\psi_r'(x_0)
    \end{array}
  \right),
\end{equation}
where~$\psi_s$, $ \psi'_s$, $\psi_r$, and~$\psi_r'$ are defined
in~(9). The matrix~$B= \left(\begin{array}{c c} \alpha & \gamma-i\mu \\
    \gamma+i\mu & -\beta
\end{array}
\right)$ is self-adjoint and each condition in~(10)--(14) follows
from~(\ref{eq:18a}) by setting three of the four
parameters~$\alpha$, $\beta$, $\gamma$, $\mu$ to zero. For an
arbitrary self-adjoint matrix~$B$, the conditions~(\ref{eq:18a}) make
a particular case of self-adjoint boundary conditions of the form~(10)
with the unitary matrix~$\hat{U}=(B-i)^{-1}(B+i)$.

Note that the boundary conditions~(\ref{eq:18a}) do not contain all
non-splitting self-adjoint boundary conditions of the form~(4). In
particular, they do not include boundary conditions of the form
\begin{equation}\label{eq:19}
  \psi'(x_0+0)=
  i\lambda_{0}\psi(x_0-0),\,\,\,\psi'(x_0-0)=i\lambda_{0}
  \psi(x_0+0)
\end{equation}
with a real constant~$\lambda_{0}$. The boundary
conditions~(\ref{eq:19}) describe a point interaction, in the
point~$x=x_0$, transparent for the waves~$e^{i\lambda x}$
with~$\lambda=\lambda_{0}$. In this case, the function~$\psi=
e^{i\lambda_{0} x}$ satisfies the boundary
conditions~(\ref{eq:19}) and the Schr\"odinger equation. Boundary
conditions~(\ref{eq:19}) have the form~(6) with the matrix~$\Lambda= i\left(\begin{array}{c c} 0 & -\lambda_0^{-1} \\
    \lambda_{0} & 0
\end{array}
\right) $.

Let us also give a relation between the matrix~$\Lambda$ from the
boundary condition~(6) and the matrix~$B$ from the
conditions~(\ref{eq:18a}),
$$
\Lambda= \frac{1}{D}\left|\left| \begin{array}{c c}
      \theta_+& \beta \\
      \alpha & \theta_-
\end{array}
\right|\right|,
$$
where~$D=(1-\frac{i}{2}\mu)^{2}-\frac{1}{4}\alpha\beta-\frac{1}{4}\gamma^2$,
$\theta_{\pm}=(1\pm
\frac{\gamma}{2})^{2}+\frac{1}{4}\alpha\beta+\frac{1}{4}\mu^2$.

The Schr\"odinger operator~$L_{B}$ corresponding to the boundary
conditions~(\ref{eq:18a}) for a point interaction in the
point~$x_0=0$ can formally be represented with the following
expression containing the Dirac $\delta$-function and its
derivative~$\delta'(x)$,
\begin{equation}\label{eq:1d}
  L_{B}=-\frac{d^{2}}{dx^{2}}+\alpha\delta(x)(\cdot,
  \delta)-\beta\delta'(x)(\cdot,
  \delta')+(\gamma+i\mu)\delta'(x)(\cdot,
  \delta)+(\gamma-i\mu)\delta(x)(\cdot, \delta').
\end{equation}
Here the differentiation~$\frac{d^{2}}{dx^{2}}$ is understood in
the distribution sense, and the functionals~$(\cdot, \delta)$
and~$ (\cdot, \delta')$ are defined by~$(\psi,
\delta)=\psi_{r}(0)=\frac{1}{2}[\psi(+0)+\psi(-0)]$, $(\psi,
\delta')=-\psi'_{r}(0)=-\frac{1}{2}[\psi'(+0)+\psi'(-0)]$. The
domain of the operator~$L_{B}$ is defined by the
condition~$L_{B}\psi \in L_{2} (R^1)$ imposed on the
functions~$\psi$ [4].

It is well known [3,4] that a model for point interactions is
exactly solvable and can serve as a good approximation of real
Schr\"odinger operators if the potential~$v$ has small support in
a neighborhood of the point~$x_0$, that is,~$v(x)=0$ for~$|x-x_0|>
\varepsilon$, and the processes under the study have the
energy~$\lambda^{2}$ much less than~$\varepsilon^{-2}$. Here it is
assumed that, for the energies under consideration, the
matrix~$\Lambda_{\varepsilon}$ that connects values of
solutions~$\psi$ of the Schr\"odinger
equation~$[-\frac{d^{2}}{dx^{2}}+v]\psi=\lambda^{2}\psi$ and their
derivatives~$\psi'(x)$ for~$x=x_0-\varepsilon$ and
~$x_0+\varepsilon$, that is, $\col(\psi(x_0+\varepsilon),
\psi'(x_0+\varepsilon))=
\Lambda_{\varepsilon}\col(\psi(x_0-\varepsilon),
\psi'(x_0-\varepsilon))$, is close to the matrix~$\Lambda$ that
defines the boundary conditions~(\ref{eq:6}) for the point
interaction. Thus the Schr\"odinger operator with point
interaction can be considered as a limit (in a certain sense,
e.g., in the sense of uniform resolvent convergence),
as~$\varepsilon\rightarrow 0$, of Schr\"odinger operators with the
potentials~$v_{\varepsilon}(x)$
with~$\Lambda_{\varepsilon}\rightarrow\Lambda$
for~$\varepsilon\rightarrow 0$. Here, the
potentials~$v_{\varepsilon}(x)$ themselves may or may not have a
limit as~$\varepsilon\rightarrow 0$ even in the sense of
distributions. It can happen that their limit values, even if they
exist, do not determine the character and the intensity of the
point interaction.

Let us look at this phenomenon in greater details for the case of
$\delta'$-potentials; this case was considered in a number of
papers~[2, 16-19,21,28,29,31-35]. For a model of
$\delta$-potentials with intensity~$\alpha$, one can take a
sequence of regular potentials~$v_{\varepsilon}(x)\rightarrow
\alpha \delta(x)$ with~$\varepsilon\rightarrow 0$, for
example,~$v_{\varepsilon}(x)=\alpha
\varepsilon^{-1}\varphi(\frac{x}{\varepsilon})$, where the
compactly supported function~$\varphi$ is such
that~$\int\varphi(x)\,dx=1$. More complex potentials can be well
modeled on small intervals by a sum of several $\delta$-functions,
\begin{equation}\label{eq:2d}
  v_{\varepsilon}(x)=\sum_{j=1}^{N}\,
  \alpha_{j}(\varepsilon)\delta(x-x_j(\varepsilon)),
\end{equation}
where all~$x_j(\varepsilon)\rightarrow x_{0}$
for~$\varepsilon\rightarrow 0$. It is shown in~[6] that the
$\delta'$-interaction is well modeled with three approaching
$\delta$-functions that have special opposite sign increasing
intensities~$\alpha_{j}(\varepsilon)$. When modeling a
$\delta'$-potential  of intensity~$\gamma$, the number of terms in
representation~(\ref{eq:2d}) depends on the conditions to be
satisfied. Since the matrix~$\Lambda$ in the boundary
conditions~(\ref{eq:6}) is diagonal for the $\delta'$-potential of
intensity~$\gamma$, there are two necessary conditions on the
elements of the matrix~$\Lambda_{\varepsilon}$,
\begin{itemize}
\item [1)] $\lim\limits_{\varepsilon\rightarrow\ 0}(\Lambda_{\varepsilon})_{2,1}=0$,
\item [2)] $\lim\limits_{\varepsilon\rightarrow\ 0}(\Lambda_{\varepsilon})_{1,1}=
  (1+\frac{\gamma}{2})(1-\frac{\gamma}{2})^{-1}$.
\end{itemize}
These two conditions can be satisfied with two terms in
approximation~(\ref{eq:2d}),
\begin{equation}\label{eq:3d}
  v_{\varepsilon}(x)= \alpha_1
  \varepsilon^{-1}\delta(x)+\alpha_2\varepsilon^{-1}\delta(x-\varepsilon),
\end{equation}
where~$\alpha_1=\gamma (1-\frac{\gamma}{2})^{-1}$,
$\alpha_2=-\gamma (1+\frac{\gamma}{2})^{-1}$.

Here, the potentials~$v_{\varepsilon}$ do not have a limit
as~$\varepsilon\rightarrow 0$ in the sense of distributions. In
this case, the matrix~$\Lambda_{\varepsilon}$ can be written as a
product of three matrices~$\Lambda_{\varepsilon}=\Lambda_{2}
\Lambda_{\varepsilon}^{0}\Lambda_{1}$, where~$\Lambda_{j}=
\left(\begin{array}{c c} 1 & 0\\
    \alpha_j \varepsilon^{-1} & 1
\end{array}
\right)$,  $j=1,2$, $\Lambda_{\varepsilon}^{0}= \left(\begin{array}{c c} \cos \lambda\varepsilon & \frac{\sin \lambda\varepsilon}{\lambda}\\
    -\lambda \sin \lambda \varepsilon & \cos \lambda \varepsilon
\end{array}
\right)$. These matrices give a relation between the
solutions~$\psi(x)$ of the Schr\"odinger equation
$$
-\frac{d^2}{dx^2}\psi+v_{\varepsilon}\psi=\lambda^{2}\psi
$$
and its derivatives~$\psi'(x)$ in different points~$x$,
$$
\begin{array}{l}
\col(\psi(+0), \psi'(+0))=\Lambda_1\col( \psi(-0), \psi'(-0)),\\[2mm]
\col(\psi(\varepsilon-0),\psi'(\varepsilon-0))=\Lambda_{\varepsilon}^{0}\col(
\psi(+0), \psi'(+0)),\\[2mm] \col(\psi(\varepsilon+0),
\psi'(\varepsilon+0))=\Lambda_{2}\col( \psi(\varepsilon-0),
\psi'(\varepsilon-0)).
\end{array}
$$

Using the explicit form of~$\alpha_j$ we get
$$
\lim\limits_{\varepsilon\rightarrow 0}
\Lambda_{\varepsilon}=\left(\begin{array}{c c} \theta & 0\\
    0  &  \theta^{-1}
  \end{array}
\right),
$$
where~$\displaystyle \theta = \frac{2+\gamma}{2-\gamma}$. Hence,
the limit Schr\"odinger operator corresponds to a point
interaction having $\delta'$-potential of intensity~$\gamma$.

One can additionally require that~$v_{\varepsilon}(x) \rightarrow
\kappa \delta'(x)$ in~(\ref{eq:2d}) as~$\varepsilon\rightarrow 0$.
This can be achieved if we take
\begin{equation}\label{eq:4d}
  v_{\varepsilon}(x)= \alpha_{1}\varepsilon^{-1}
  \delta(x+\varepsilon)+\alpha_{2}\varepsilon^{-1}
  \delta(x)+\alpha_{3}\varepsilon^{-1} \delta(x-\varepsilon)
\end{equation}
in~(\ref{eq:2d}), where~$\alpha_{2}= \pm 2\gamma[
\gamma^{2}-4]^{-\frac{1}{2}}$,
$\alpha_{1}=\frac{\gamma}{2}+\frac{\alpha_{2}}{2}(1+\frac{\gamma}{2})$,
$\alpha_{3}=-\frac{\gamma}{2}-\frac{\alpha_{2}}{2}(1-\frac{\gamma}{2})$.

In the limit as~$\varepsilon\rightarrow 0$, the Schr\"odinger
operators with the potentials~$v_{\varepsilon}(x)$ of the
form~(\ref{eq:4d}) define point interaction of $\delta'$-potential
type with intensity~$\gamma$, and the limit~$v_{\varepsilon}(x)
\rightarrow \kappa \delta'(x)$ exists in the distribution sense,
where the
constant~$\kappa=\alpha_{1}-\alpha_{3}=\gamma+\alpha_{2}$ depends
on the choice of the sign of~$\alpha_{2}$ and, consequently, it
does not determine the intensity~$\gamma$. Moreover, considering
an expression of the form~(\ref{eq:2d}) for the
potentials~$v_{\varepsilon}(x)$ with four terms
\begin{equation}\label{eq:5d}
  v_{\varepsilon}(x)= \alpha_{1}\varepsilon^{-1}
  \delta(x)+\alpha_{2}\varepsilon^{-1}
  \delta(x-\varepsilon)+\alpha_{3}\varepsilon^{-1}
  \delta(x-2\varepsilon)+ \alpha_{4}\varepsilon^{-1}
  \delta(x-3\varepsilon),
\end{equation}
where~$\alpha_{1}=-1$, $\alpha_{2}=6$, $\alpha_{3}=-3 $,
$\alpha_{4}=-2$ we obtain~$\lim\limits_{\varepsilon \rightarrow 0}
v_{\varepsilon}(x) = 6 \delta'(x)$ in the sense of distributions.
On the other hand, it is easy to see
that~$\lim\limits_{\varepsilon
  \rightarrow 0} \Lambda_{3 \varepsilon}=I$, that is, if~$\varepsilon
\rightarrow 0$, the Schr\"odinger operators with
potentials~(\ref{eq:5d}) converge to a free Schr\"odinger
operator. By taking~$\alpha_{1}=\alpha_{4}=3$,
$\alpha_{2}=\alpha_{3}=-3$ in~(\ref{eq:5d}), we have~$
v_{\varepsilon}(x)\rightarrow 0$ and the Schr\"odinger operators
converge to a direct sum of operators on the
spaces~$L_2(-\infty,0)$ and~$L_2(0, +\infty)$ corresponding to the
Dirichlet conditions~$\psi(\pm 0 )=0$.

Let us remark that if the Schr\"odinger operators have potentials
in the form of~(\ref{eq:3d})---(\ref{eq:5d}), then the kernels of
the resolvents for these operators can be written explicitly
similarly to the case of the limit Schr\"odinger operator. This
yields that these operators converge, as~$\varepsilon\rightarrow
0$, in the sense of uniform resolvent convergence.

The above conclusions about Schr\"odinger operators with
potentials~(\ref{eq:3d})-- (\ref{eq:5d}) remain also true
if~$v_{\varepsilon}$ are piecewise constant or
even~$v_{\varepsilon}\in C_0^{\infty}(R^1)$ if they can well
approximate each term in~(\ref{eq:3d})---(\ref{eq:5d}).

Let us also make a remark on one more feature of point
interactions. If the support of the potential~$v_{\varepsilon}(x)$
belongs to the interval~$(-\varepsilon, \varepsilon)$ and its
components~$v^{-}_{\varepsilon}(x) = \theta
(-x)v_{\varepsilon}(x)$, $ v_{\varepsilon}^{+}(x)= \theta (x)
v_{\varepsilon}(x)$, where~$\theta$ is the unit Heaviside
function, determine point interactions with the corresponding
matrices~$\Lambda^{-}$ and~$\Lambda^{+}$,
as~$\varepsilon\rightarrow 0$, then the
potential~$v_{\varepsilon}(x)$ also gives rise to a point
interaction, as~$\varepsilon\rightarrow 0$, with the
matrix~$\Lambda=\Lambda^{+} \Lambda^{-}$. This leads to additivity
of intensities~$\alpha$ and~$\beta$ for $\delta$- and
$\delta'$-interactions, since they correspond to triangular
matrices~$\Lambda^{-}$, $\Lambda^{+}$, $\Lambda$. For point
interactions with $\delta'$-type potentials and $\delta$-magnetic
potentials, the intensities~$\gamma$ and~$\mu$ do not have such an
additivity property. Here, if~$\gamma_{-}$ and~$\gamma_{+}$ are
intensities of $\delta'$-potentials corresponding
to~$v_{\varepsilon}^{-}$ and~$v_{\varepsilon}^{+}$, then the total
intensity~$\gamma$ is found as~$\gamma=(\gamma_{-}+\gamma_{+}
)(1+\frac{1}{4}\gamma_{-}\gamma_{+})^{-1}$.  Thus, for point
interactions with $\delta'$-type potential and $\delta$-magnetic
potential, the ``additive'' characteristics of the intensities are
useful. The additive characteristic~$\xi$ for $\delta'$-potential
with intensity~$\gamma$ are defined by the
identities~$\frac{2+\gamma}{2-\gamma}=\pm e^{\xi_{\pm}}$, where
the sign ``$+$'' is taken if~$|\gamma|< 2$ and we take the sign
``$-$'' if~$|\gamma|> 2$. A more exact definition of additive
characteristic for point interactions with $\delta'$-potential is
the following. Additive characteristic is a pair~$(\xi,s)$
consisting of the number~$\xi$ and the sign~$s=\pm 1$. As
two-point interactions with $\delta'$-potentials having
characteristics~$(\xi_{1},s_{1})$ and~$(\xi_{2},s_{2})$ approach,
the total characteristic~$(\xi,s)$ is found
as~$(\xi,s)=(\xi_{1}+\xi_{2},s_{1}\cdot s_{2})$, which corresponds
to the above ``adding'' rule for the intensities~$\gamma_{-}$
and~$\gamma_{+}$.

For a point interaction with $\delta$-magnetic potential of
intensity~$\mu$, the $\Lambda$-matrix in the boundary
condition~(\ref{eq:6}) is a multiple of the identity
matrix,~$\Lambda= e^{i\eta}I$. Hence, it is convenient to take the
number~$\eta$ to be an ``additive'' characte\-ristic of the
~$\delta$-magnetic potential. There is a relation between~$\mu$
and~$\eta$, $\mu=2\tan\frac{\eta}{2}$. For two approaching point
interactions with $\delta$-magnetic potentials having
characteristics~$\eta_{1}$ and~$\eta_{2}$, the corresponding total
characteristic is~$\eta= \eta_{1}+\eta_{2}$.

It is not true that if the Schr\"odinger
operators~$-\frac{d^{2}}{dx^{2}}+v_{\varepsilon}(x)$ converge,
as~$\varepsilon\rightarrow 0$, to a Schr\"odinger operator with
point interaction of a certain type then the
operators~$-\frac{d^{2}}{dx^{2}}+k v_{\varepsilon}(x)$,
where~$k\neq 1$ is an arbitrary real constant, also converge to a
Schr\"odinger operator with point interaction of the same type. In
the general case, this is true only for $\delta$-potential. It is
shown in~[21] that, for special approximations of
$\alpha\delta'$-functions
where~$v_{\varepsilon}=\alpha\varepsilon^{-2}\psi(\frac{x}{\varepsilon})$,
$\int\,\psi(x)\,dx=0$, $\int\,x \psi(x)\,dx=-1$, the Schr\"odinger
operators have a limit that defines a point interaction of
$\delta'$-potential only for special ``resonance'' values
of~$\alpha$.

\begin{proposition}
  For a one-dimensional Schr\"odinger operator~$A$ with local
  interactions on a finite set~$X=\{ x_1,...,x_n\}$, to describe a
  $\delta'$-interaction it is necessary and sufficient that all the
  functions~$\chi(x) \in C_0^{\infty}(R^1) $ such that~$\chi'(x)\in
  C_0^{\infty}(R^1 \setminus X) $ belong to the domain of the
  operator~$A$ and the operator~$A$ does not admit a representation as a
  direct sum~$A=A_1 \oplus A_2$ of two self-adjoint operators on the
  spaces~$L_2(-\infty, a)$ and~$L_2(a,+\infty)$ for any~$a$.
\end{proposition}

\begin{proof}
  Necessity follows, since the boundary conditions for a
  $\delta'$-interaction can not be represented in the
  form~(\ref{eq:5}), that is, the operator with $\delta'$-interaction
  can not be represented as~$A=A_1 \oplus A_2$. Moreover, each
  function~$\chi(x) \in C_0^{\infty}(R^1) $ satisfying~$\chi'(x)\in
  C_0^{\infty}(R^1 \setminus X) $ assumes constant values in small
  neighborhoods of the points~$x_k \in X$. Hence, this function
  satisfies the boundary conditions for $\delta'$-interaction on the
  set~$X$ with arbitrary intensities.

  Sufficiency follows, since if the operator~$A$ does not admit the
  representation~$A=A_1 \oplus A_2$ on the space~$L_2(-\infty, a)
  \oplus L_2(a,+\infty)$ and the function~$\chi(x) \in {\mathfrak
    D}(A)$ is distinct from zero only in a small neighborhood of the
  point~$x_k $ not containing other points of~$X$, the boundary
  condition~(\ref{eq:6}) leads to the matrix~$\Lambda= \left(\begin{array}{c c} 1 & \beta_k \\
      0 & 1
\end{array}
\right)$ with real~$\beta_k$. This corresponds to
$\delta'$-interaction in the point~$x_k$ with intensity~$\beta_k$.
\end{proof}

\section{Interactions on a set of measure zero}\label{S:3}

Let $\Gamma$ be a closed bounded subset of $R^1$ of Lebesgue
measure zero, $|\Gamma|=0.$ There is a symmetric minimal operator
$L_{min,
  \Gamma}$ defined on the space $L_2(R^1)$ by $L_{min, \Gamma
}\varphi(x)=-\varphi''(x)$ on functions $\varphi \in
C_0^{\infty}(R^1\setminus \Gamma)$. An operator adjoint
in~$L_2(R^1)$ to the operator~$L_{min, \Gamma}$ is maximal. Its
domain is~${\mathfrak D}(L_{\max, \Gamma})=W_2^2(R^1 \setminus
\Gamma)$.

\sloppy Each self-adjoint operator $A$, that is, a self-adjoint
extension of the operator~$L_{min, \Gamma }$, defines an
interaction on the set~$\Gamma$.

\fussy
\begin{definition}\label{D:1}
  We will say that a self-adjoint operator~$A \supset L_{min, \Gamma}$
  defines a local interaction on~$\Gamma$ if~$u(x) \in {\mathfrak
    D}(A)$ implies that~$\chi(x) u(x) \in {\mathfrak D}(A)$ for an
   arbitrary cutting function~$\chi(x) \in C_0^{\infty}(R^1)$ such
  that~$\chi'(x)\in C_0^{\infty}(R^1 \setminus \Gamma)$. In this case,
  we also say that the functions in~${\mathfrak D}(A)$ satisfy local
  boundary conditions.
\end{definition}
\begin{lemma}\label{L:1}
  Let~$A_{\Gamma}$ be a self-adjoint operator on the space
  $L_2(R^1)$  describing a local interaction on~$\Gamma$. Let
$a,b \not\in
  \Gamma,$\, $a< b.$  Then, the second Green formula holds true
  for any functions $f,g \in {\mathfrak D}(A_{\Gamma})$
\begin{equation}\label{eq:26d}
    \int\limits_{a}^b[(A_{\Gamma} f)(x)\overline{g(x)}-
    f(x)\overline{(A_{\Gamma}
    g)(x)}]\, dx=
    f(b)\overline{g'(b)}-f'(b)\overline{g(b)}-f(a)\overline{g'(a)}+f'(a)\overline{g(a)}.
  \end{equation}
  In case of $a=-\infty$ or $b=+\infty$ in the rightside of~(\ref{eq:26d})
  there are no terms of boundary data of functions $f,\,g$ at
  points $a=-\infty$ or $b=+\infty$.
\end{lemma}
\begin{proof}
Let $a_-<a$ and $b_+> b$ be such that the intervals $(a_-,a)$ and
$(b, b_+)$ contain no points of $\Gamma$.  Let $\varphi(x)\in
C_0^{\infty}(a_-,b_+)$ be a cutting function that equals to 1 with
$x \in (a,b)$.  The functions $f_0= \varphi \cdot f$ and $g_0= \varphi
\cdot g$  belong to domain of the operator $A_{\Gamma}$ since the
operator $A_{\Gamma}$ describes a local interaction on $\Gamma$
according to definition 1. The functions $f_0$,\ $g_0$,\,
$A_{\Gamma} f_0$,\ $A_{\Gamma} g_0$ coincide with $f$,\ $g$,\,
$A_{\Gamma} f$,\ $A_{\Gamma} g$ at $x \in (a,b)$, respectively.
Therefore the righthandside of~(\ref{eq:26d}) can be written in the
following form:
$$
\begin{array}{cl}
    \int\limits_{a}^b[(A_{\Gamma} f_0)(x)\overline{g_0(x)}-
    f_0(x)\overline{(A_{\Gamma}
    g_0)(x)}]\, dx= \\ [3mm] \int\limits_{a_-}^{b_+}\,[(A_{\Gamma} f_0)(x)\overline{g_0(x)}-
    f_0(x)\overline{(A_{\Gamma}
    g_0)(x)}]\, dx -
 \int\limits_{(a_-,a) \cup (b,b_+)}\,[-f_0''\overline{g_0}+f_0\overline{g_0''}]\, dx =\\ [4mm]
 (A_{\Gamma} f_0,
    g_0)-(f_0, A_{\Gamma} g_0)+
    f_0(b)\overline{g_0'(b)}-f_0'(b)\overline{g_0(b)}-f_0(a)\overline{g_0'(a)}+f_0'(a)\overline{g_0(a)}.
     \end{array}
 $$
This leads to equality~(\ref{eq:26d}) since the operator
$A_{\Gamma}$  is a self--adjoint on $L_2(R^1)$ and the functions
$f_0$,\ $g_0$ coincide with $f$,\ $g$ on the interval $[a, b]$.

 In case of $a=-\infty$, taking of $a_-=-\infty$, we do not have
 terms with boundary data of functions $f$ and  $g$ at the point
 $a=-\infty$. In the same way, in the case  when $b=+\infty$, we do not
 have terms of values $f$ and  $g$ at the point $b=+\infty$.
 Let us note that the boundary data $\Gamma_f=( f(a),
 f(a)',f(b),f(b)')$ of functions $f \in {\mathfrak D}(A_{\Gamma})$
 fill all the space $E^4.$
\end{proof}

\begin{proposition}
Let the conditions of Lemma 1 hold. Let $A_{\Gamma}^{(a, b)}$ be a
restriction of the operator $A_{\Gamma}$ on the space $L_2(a,b)$
acting as following $A_{\Gamma}^{(a,
b)}[\chi_{(a,b)}f]=\chi_{(a,b)} A_{\Gamma} f$  for any function $f
\in {\mathfrak D}(A_{\Gamma})$ such that $f(a)=f(b)=0.$ Here,
$\chi_{(a,b)} (x)$ is a characteristic function on the interval
$(a,b)$, i.e. $\chi_{(a,b)} (x)=1$ with $ x \in (a,b)$, and
$\chi_{(a,b)} (x)=0$ with $ x \notin (a,b).$  Then,
$A_{\Gamma}^{(a, b)}$ is a self--adjoint operator on the $L_2(a,
b)$.
\end{proposition}
\begin{proof}
The definition of the operator $A_{\Gamma}^{(a, b)}$ is correct,
since  the operator $A_{\Gamma}$ is a local operator $A_{\Gamma}
\psi(x)=-\psi''(x),$\,$x \notin \Gamma.$  Formula~(\ref{eq:26d})
shows that the operator $A_{\Gamma}^{(a, b)}$ is a symmetric
restriction of maximal operator for functions that correspond to
self--adjoint boundary conditions $f(a)=f(b)=0.$ Therefore,
$A_{\Gamma}^{(a, b)}$ is a self--adjoint operator on $L_2(a,b)$.
\end{proof}
\begin{lemma}\label{L:2}
  Let a self-adjoint operator~$A_{\Gamma}$ define a local interaction
  on~$\Gamma$. Let~$x_0 \in \Gamma$ be an isolated point of the
  set~$\Gamma$. Then the functions in~${\mathfrak D}(A_{\Gamma})$ satisfy local
  boundary conditions~(4) in the point~$x_0$.
\end{lemma}

\begin{proof}
  Let~$x_0 \in \Gamma$ be an isolated point of the set~$\Gamma$. Then
  there exists an interval $(a, b)$ such that $(a, b)$ contains no
  other points of the set $\Gamma$ except for $x_0.$
Let us consider all functions $\psi,\varphi \in {\mathfrak
D}(A_{\Gamma})$ that equal to zero at $x=a,b$. Then,
$$(A_{\Gamma}^{(a, b)}\psi, \varphi)-(\psi,(A_{\Gamma}^{(a, b)}\varphi
)=\int\limits_{a}^b\, [-\psi''(x)\overline{\varphi(x)}-
\psi(x)\overline{\varphi''(x)} ]\,dx=\omega(\Gamma\psi,
\Gamma\varphi),
$$
where the  boundary form $\omega$ is defined in (2). Since the
operator $A_{\Gamma}^{(a, b)}$ is a self--adjoint operator on
$L_2(a,b)$ in virtue of Proposition 2 then functions $\psi \in
{\mathfrak D}(A_{\Gamma}^{(a, b)})$ must satisfy the boundary
condition (4).

\end{proof}

\begin{definition}\label{D:2}
  We will say that a self-adjoint operator~$A$ admits splitting
  boundary conditions in a point~$x_0 \in \Gamma$ if the operator~$A$
  on the space~$L_2(R^1)=L_2(-\infty, x_0 )\oplus L_2(x_0,+\infty)$
  admits a representation in the form of the direct sum~$A=A_1 \oplus
  A_2$ of a self-adjoint operator~$A_1$ on the space~$L_2(-\infty, x_0
  )$ and an operator~$A_2$ on the space~$L_2(x_0,+\infty)$.
\end{definition}

\begin{lemma}\label{L:3}
  Let~$x_0 \in \Gamma$ be an isolated point of the set~$\Gamma$. Let a
  self-adjoint operator~$A_{\Gamma}$ define a local interaction on~$\Gamma$
  and  not admit splitting boundary conditions in a
  point~$x_0$. Then the functions in~${\mathfrak D}(A)$ satisfy
  non-splitting boundary conditions in the point~$x_0$ of the form~(6).
\end{lemma}

\begin{proof}
  The proof follows from Lemma~2 and the general form of~(6) for
  non-splitting boundary conditions.
\end{proof}

\begin{definition}\label{D:3}
  We say that a self-adjoint operator~$A\supset L_{min, \Gamma}$
  describes a $\delta'$-interaction on~$\Gamma$ if the operator~$A$
  corresponds to local non-splitting boundary conditions on~$\Gamma$
  and all the functions~$\chi(x)\in C_0^{\infty}(R^1)$
  satisfying~$\chi'(x)\in C_0^{\infty}(R^1 \setminus \Gamma)$ belong
  to~${\mathfrak D}(A)$.
\end{definition}

\begin{lemma}\label{L:3}
  Let a self-adjoint operator~$A\supset L_{min, \Gamma}$ define a
  $\delta'$-interaction on~$\Gamma$. If~$x_0 \in \Gamma$ is an
  isolated point of the set~$\Gamma$, then there exists a real
  number~$\beta$ such that the functions in~${\mathfrak D}(A)$ satisfy
  boundary conditions~(12) for point $\delta'$-interactions with
  intensity~$\beta$.
\end{lemma}

\begin{proof}
  The proof follows since a function~$\varphi_0(x)$, that equals~$1$ on a
  small neighborhood of the point~$x_0$, belongs to~${\mathfrak
    D}(A)$, and satisfies the local non-splitting boundary
  condition~(6) if and only if the boundary condition describes a
  $\delta'$-interaction (see the proof of Proposition~1).
\end{proof}

\section{Test functions for $\delta'$--interactions}\label{S:4}

Let~$x_0 \in \Gamma$ be an isolated point of a bounded closed
set~$\Gamma$ having Lebesgue measure zero. Let a self-adjoint
operator~$A$ define a $\delta'$-interaction on~$\Gamma$. In
particular, the functions~$\psi(x) \in {\mathfrak D}(A)$, in the
point~$x_0$, satisfy the boundary conditions
\begin{equation}\label{eq:28d}
  \psi'(x_0+0)=\psi'(x_0-0),\qquad \psi(x_0+0)-\psi(x_0-0)=\beta
  \frac{1}{2}[\psi'(x_0+0)+\psi'(x_0-0)]
\end{equation}
where~$\beta$ is intensity of the $\delta'$-interaction in the
point~$x_0$. We construct a function that belongs to~${\mathfrak
  D}(A)$, has compact support, satisfies condition~(\ref{eq:28d}) in
the point~$x_0$, and consists piecewise of parabolas and
constants.

\begin{definition}\label{D:4}
  Consider the following test function that depends on $4$
  parameters~$\varepsilon,\,\beta,\,l,\,r$:
  $$
  t(x,\varepsilon,\beta,l,r)=\left\{
    \begin{array}{ll}
      0,  &  x \leq -\varepsilon,\\[2mm]
      \frac{1}{2\varepsilon}(x+\varepsilon)^2,   & -\varepsilon \leq x < 0, \\[2mm]
      \beta+ \varepsilon -\frac{1}{2\varepsilon}(x-\varepsilon)^2,  & 0 < x \leq \varepsilon, \\[2mm]
      \beta+\varepsilon,   &  \varepsilon \leq x \leq l, \\[2mm]
      \beta+\varepsilon - \frac{\beta+\varepsilon}{2r^2} (l-x)^2,  &
      l\leq x \leq l+r,\\[2mm]
      \frac{\beta+\varepsilon}{2r^2} (l+2r- x)^2,  & l+r \leq x \leq
      l+2r,\\[2mm]
      0,  & l+2r \leq x.
    \end{array}
  \right.
  $$
\end{definition}
\begin{proposition}
The following $4$ properties of test functions easily follow from
Definition~4.
\begin{itemize}
\item[1)] A test function is twice (weakly) differentiable with
  compact support for~$x\neq 0$.
\item[2)] The test
  function~$\hat{t}(x)=t(x-x_0;\varepsilon,\beta,l,r)$ satisfies
  condition~(27) for a $\delta'$-interaction with intensity~$\beta$.
\item[3)] If~$0 < \varepsilon \leq \varepsilon_0$ is such that the $2
  \varepsilon_0$-neighborhood of the point~$x_0$ contains no points
  of~$\Gamma$ other than~$x_0$ and the value of~$l$ is larger than the
  diameter of ~$\Gamma$, then the test
  function~$\hat{t}(x)=t(x-x_0;\varepsilon,\beta,l,r)$ belongs to the
  domain of any self-adjoint operator~$A$ defining a
  $\delta'$-interaction on~$\Gamma$ and the$\delta'$-interaction in the
  point~$x_0$ with intensity~$\beta$.
\item[4)] If 3) holds, then~$\displaystyle (A\hat{t},
  \hat{t})=\beta+\frac{2}{3}\varepsilon
  +\frac{2}{3r}(\beta+\varepsilon)^2$.
\end{itemize}
 \end{proposition}

\section{Number of negative eigenvalues for
  $\delta'$-inter\-action}\label{S:5}

It is well known~[7, 23] that for point $\delta'$-interactions
at finitely many points
the
number of negative eigenvalues of the Schr\"odinger operator
equals the number of points having negative intensities of the
$\delta'$-interactions.
In the case of infinitely many points it can happen that the point spectrum
is empty, cf. \cite{AlGHH}, Theorem 3.6.
 However, for bounded $\Gamma$ and if the negative spectrum is discrete,
there is the
following generalization of the mentioned result on the number of
negative eigenvalues.

\begin{theorem}\label{Th:1}
Let~$A_{\Gamma,\delta'}$ be a self-adjoint Schr\"odinger operator
on~$L_2(R^1)$ with
  $\delta'$-interaction on a closed bounded set~$\Gamma$ of Lebesgue
  measure zero. Let the negative spectrum of
  the operator~$A_{\Gamma,\delta'}$ be discrete. Then
  the number of negative eigenvalues of the operator~$A_{\Gamma,\delta'}$ is not less than the
  number of isolated points of the set~$\Gamma$ having negative values
  of intensities of the $\delta'$-interactions.
\end{theorem}

\begin{proof}

   Let~$x_1,...,x_n$ be isolated points of the
  set~$\Gamma$ having negative values $\beta_k < 0$ of intensities
  of $\delta'$-interactions in the points~$x_k$,
  $k=1,...,n$. Let~$\varepsilon_0 >0$ be a sufficiently small number
  such that the $2 \varepsilon_0$-neighborhood of each point~$x_k \in
  \Gamma$ contains no points of the set~$\Gamma$ other
  than~$x_k$. Let~${\mathcal L}_n$ be an $n$-dimensional subspace
  of~${\mathfrak D}(A)$ containing the
  test-functions~$\hat{t}(x)=t(x-x_k;\beta_k, \varepsilon_k,l_k,r_k)$,
  $k=1,...,n$, corresponding to the points~$x_k$ with the
  intensities~$\beta_k < 0$. Choose numbers~$\varepsilon_k \leq
  \varepsilon_0$ and~$r_k$ such that
  \begin{equation}\label{eq:28}
    \beta_k+\frac{2}{3}\varepsilon_k
    +\frac{2}{3r_k}(\beta_k+\varepsilon_k)^2=\frac{1}{2}\beta_k < 0.
  \end{equation}
  Moreover, choose all~$l_k \geq l$, where~$l$ is larger than the diameter
of $\Gamma$, and such that the
  intervals~$I_k=(l_k,l_k+2r_k)$, $k=1,...,n$, do not intersect for
  distinct~$k$. Hence, every function~$u \in {\mathcal L}_n$ can be
  represented as
  \begin{equation}\label{eq:29}
    u(x)=\sum\limits_{k=1}^{n}\, a_k t_k(x-x_k;\beta_k,
    \varepsilon_k,l_k,r_k) ,
  \end{equation}
  where~$a_k$ are complex constants. Using properties of test
  functions and~(\ref{eq:28}) it is easy to see that the quadratic
  form~$(Au,u)$ is negative definite on the $n$-dimensional
  subspace~${\mathcal L}_n$,
i.e. for $u\in {\cal L}_n \setminus \{ 0 \}$  we have
  \begin{equation}\label{neq:}
    (Au,u)=\sum\limits_{k=1}^{n}\,|
    a_k|^2(At_k,t_k)=\frac{1}{2}\sum\limits_{k=1}^{n}\,\beta_k| a_k|^2
    < 0.
  \end{equation}
  Hence, it follows from the variational minimax principle~[26] that the
  operator~$A$ has at least~$n$ negative eigenvalues.
\end{proof}

\section{Boundary conditions for $\delta'$-interactions}\label{S:6}

If~$\Gamma=X$ is a finite or countable set of
points,~$X=\{x_k\}_{k=1}^{\infty}$, then the Schr\"odinger
operator~$L_{X, \beta}$ with $\delta'$-interaction in the
points~$x_k \in X$ with intensities~$\beta_{k}$ is defined on
functions that belong to the space~$W_2^2(R^1 \setminus X)$ and
satisfy the boundary conditions~(12) in every point~$x=x_k$.

Let~$\Gamma$ be a closed bounded subset of~$R^1$ of measure
zero,~$|\Gamma|=0$. The Schr\"odinger operator with
$\delta'$-interaction on~$\Gamma$ is defined in an abstract form
in Section~3 (see Definition~3). We will give a concrete
construction of such operators following~[8, 24].

Let~$\Gamma$ be endowed with a Radon measure, that is, a finite
regular Borel measure~$\mu$ [30] such that its support coincides
with~$\Gamma$. In this case, one can define boundary data
on~$\Gamma$ for some functions~$\psi \in W_2^2(R^1 \setminus
\Gamma)$, which is an analogue of~$\psi_{s}(x_{0})$,
$\psi'_{s}(x_{0})$, $\psi_{r}(x_{0})$, $\psi'_{r}(x_{0})$ given
in~(\ref{eq:13}).

Let a function~$\psi(x)$ and its derivative~$\psi'(x)$ have the
following representations for~$x,s \in R^1\setminus \Gamma$:
\begin{equation}\label{eq:**}
  \begin{array}{l}
    \psi(x)=\psi(s)+\int_s^{x}\,\psi'(\xi)\,d\xi+\int_{(s,x)}\,f(\xi)\mu(d\xi),\\[2mm]
    \psi'(x)=\psi'(s)+\int_s^{x}\,\psi''(\xi)\,d\xi+\int_{(s,x)}\,g(\xi)\mu(d\xi),
  \end{array}
\end{equation}
where~$f$ and~$g$ are defined on~$\Gamma$ and absolutely
integrable with respect to the measure~$\mu$. The functions~$f$
and~$g$ are called derivatives of the functions~$\psi(x)$
and~$\psi'(x)$ with respect to the measure~$\mu$, and are denoted
by~$f=\frac{d\psi}{d\mu}$, $g=\frac{d\psi'}{d\mu}$. They are
analogues of the jump functions~$\psi_{s}(x_{0})$
and~$\psi'_{s}(x_{0})$. It follows from~(\ref{eq:**}) that there
exist functions~$\psi_r(x)=\frac{1}{2}[\psi(x+0)+\psi(x-0)]$
and~$\psi'_r(x)=\frac{1}{2}[\psi'(x+0)+\psi'(x-0)]$ on~$\Gamma$
that are essentially bounded on~$\Gamma$, i.e., belong to the
space~$L_{\infty}(\Gamma, d\mu)$. All four functions~$\psi_r$,
$\psi'_r$, $\frac{d\psi}{d\mu}$, and~$\frac{d\psi'}{d\mu}$ define
boundary data on~$\Gamma$ for functions~$\psi$ that admit
representation~(\ref{eq:**}). The set of all functions in the
space~$W_2^2(R^1\setminus \Gamma)$ satisfying boundary conditions
will be denoted by~$W_2^2(R^1\setminus \Gamma; d\mu)$. For
functions~$\psi,\,\varphi \in W_2^2(R^1\setminus \Gamma; d\mu)$,
it was proved in~[8] that Green's first and second formulas hold
with boundary values of~$\psi$ and~$\varphi$ on~$\Gamma$.

Green's first formula is
\begin{equation}\label{eq:31a}
  (-\psi'', \varphi)_{L_{2}(R^1)} = (\psi',
  \varphi')_{L_{2}((R^1)} +\int\limits_{\Gamma}\,\Bigr[
  \frac{d\psi'}{d\mu}\overline{\varphi_r} +
  \psi'_r\frac{d\overline{\varphi}}{d\mu} \Bigl]d\mu.
\end{equation}
 Green's second formula is
\begin{multline}
\label{eq:32}
    (-\psi'', \varphi)_{L_{2}(R^1)} -  (\psi, -
    \varphi'')_{L_{2}((R^1)} =\int\limits_{\Gamma}\,\Bigr[
    \frac{d\psi'}{d\mu}\overline{\varphi}_r +
    \psi'_r\frac{\overline{d\varphi}}{d\mu}-\psi_r\frac{\overline{d\varphi'}}{d\mu}-\frac{d\psi}{d\mu}\overline{\varphi'_r}
    \Bigl]d\mu
    \\[2mm]
    =\omega(\Gamma\psi,\Gamma\varphi)=<\hat{\Gamma_{1}}\psi,\hat{\Gamma_{2}}\varphi>-<\hat{\Gamma_{2}}\psi,\hat{\Gamma_{1}}\varphi>,\,
 \\   \hat{\Gamma_{1}}\psi
    =\col(\frac{d\psi'}{d\mu},\frac{d\psi}{d\mu}),\qquad
   \hat{\Gamma}_2=\col(\psi_r,-\psi'_r).
\end{multline}

Green's second formula allows to consider different self-adjoint
boundary conditions that are similar to one-point conditions
considered in Section~2. They include the following boundary
conditions that correspond to $\delta'$-interaction on~$\Gamma$:
\begin{equation}\label{eq:33}
\frac{d\psi'(x)}{d\mu}=0,\,\,\frac{d\psi(x)}{d\mu}=\beta(x)\psi'_{r}(x),\,\,x\in
\Gamma.
\end{equation}
Here, the real-valued function~$\beta$ is defined on~$\Gamma$ and
is absolutely integrable with respect to measure~$\mu$. The
function~$\beta$ defines the intensity of the $\delta'$-interaction
on~$\Gamma$.

\section{Spectral properties of Schr\"odinger operator \\with
  $\delta'$-interaction}\label{S:7}

The boundary conditions~(\ref{eq:33}) define a Schr\"odinger operator
with $\delta'$-interaction on~$\Gamma$. The definition domain of
such an operator~$L_{\Gamma, \beta}$ consists of all functions in
the space~$W_2^2(R^1\setminus \Gamma; d\mu)$ that satisfy the
boundary conditions~(\ref{eq:33}). The operator acts on such a
function~$\psi$ by~$L_{\Gamma, \beta}\psi=-\psi''(x)$,
$x\not\in \Gamma$. This operator is Hermitian in virtue of
Green's second formula~(\ref{eq:32}). It was proved in~[8] that it
is self-adjoint. This is the following result.

\begin{theorem}\label{Th:2}
  Let~$\Gamma$ be a bounded closed subset of the real line, having
  Lebesgue measure zero. Let a real-valued function~$\beta$ be
  absolutely integrable on~$\Gamma$ with respect to a Radon
  measure~$\mu$. The  Schr\"odinger operator~$L_{\Gamma, \beta}$ is
  self-adjoint on the space~$L_{2}(R^1)$ and defines a
  $\delta'$-interaction on~$\Gamma$.
The negative spectrum of
  the operator~$L_{\Gamma, \beta}$ is discrete.
\end{theorem}

\begin{proof}
  Since we work here with the abstract Definition~\ref{D:3} of a
  Schr\"odinger operator with $\delta'$-interaction on~$\Gamma$, the
  proof from~[8] needs to be modified in view of this definition.
  The Schr\"odinger operator~$L_{\Gamma, \beta}$ is self-adjoint. This
  is proved in~[8] for~$\Gamma$ being a Cantor set and a
  Hausdorff measure on~$\Gamma$.
This proof is correct for the general case of
$\delta'$--interaction on a set $\Gamma$ with a measure $\mu$. Let
us consider an operator $L_{\Gamma, \beta}^{(a,b)}$ in a space
$L_{2}(a, b),$ where the interval $(a, b)$  contains the set
$\Gamma$. The domain of operator $L_{\Gamma, \beta}^{(a,b)}$
consists of the restrictions on the interval $(a, b)$ of all functions
of $W_2^2(R^1\setminus \Gamma, d\mu ),$ that satisfy boundary
conditions~(\ref{eq:33}) and also boundary conditions
$\psi(a)=0,\,\, \psi'(b)=0$ at the endpoints of interval. The
action of  the operator  $L_{\Gamma, \beta}^{(a,b)}$ on these
functions $\psi$ leads to $-\psi''(x)$ with $x\not\in
\Gamma$. Let us show that the operator $L_{\Gamma, \beta}^{(a,b)}$
is self--adjoint in the space $L_{2}(a, b).$ For this, at the
beginning, let us show that the range of values of operator
$L_{\Gamma, \beta}^{(a,b)}$ is the whole space $L_{2}(a, b).$ In
fact, since $\frac{d\psi'(x)}{d\mu}=0,$  then because
of~(\ref{eq:**}) and the boundary condition $\psi'(b)=0$:
$\psi'(x)=\int\limits_x^b\,h(s)\,ds$, where $h(x)=L_{\Gamma,
\beta}^{(a,b)}\psi(x)=-\psi''(x).$  Therefore, for any $h \in
L_{2}(a, b)$ we have $\psi'(x)=\psi_r'(x),$ and
considering~(\ref{eq:**}), boundary conditions~(\ref{eq:33}) and
condition $\psi(a)=0$ with $x\not\in  \Gamma,$ we have
\begin{equation}\label{eq:35**}
\psi(x)=\int\limits_a^x \,\psi'(s)\,ds + \int\limits_a^x
\beta(s)\psi'(s)\,d\mu(s)=\int\limits_a^b\,\mathcal{G}(x,s)h(s)\,ds
\end{equation}
where $\mathcal{G}(x,s)=\min(x,s)-a  +
\int\limits_a^{\min(x,s)}\,\beta(\xi)d\mu(\xi).$ The
representation~(\ref{eq:35**}) shows that the operator $[L_{\Gamma,
\beta}^{(a,b)}]^{-1}$ is an integral bounded Hermitian operator in
the space $L_{2}(a, b).$
 Therefore, the operator $L_{\Gamma, \beta}^{(a,b)}$   is self--adjoint in
 the space $L_{2}(a,b).$

Let us consider the direct sum of self--adjoint operators $L_{D},$\,
$L_{\Gamma, \beta}^{(a,b)},$ \,$L_{N}:$\, $L=L_{D}\oplus
L_{\Gamma, \beta}^{(a,b)}\oplus L_{N}$ in the space
$L_{2}(R^1)=L_{2}(-\infty, a) \oplus L_{2}(a, b) \oplus L_{2}( b,
\infty).$ Here, the self--adjoint operator $L_{D}$ is defined in the
space  $L_{2}(-\infty, a)$ by the differential expression
 $-\frac{d^2}{dx^2}$ on functions of the space $W_2^2(-\infty,
 a),$ that satisfy Dirichlet boundary condition $\psi(a)=0$.
 The self--adjoint operator $L_{N}$ is defined in the
space  $L_{2}( b, \infty)$  by differential expression
 $-\frac{d^2}{dx^2}$ and Neumann boundary condition  $\psi'(b)=0$.
  The self--adjoint operator $L_{\Gamma, \beta}^{(a,b)}$   is defined above in
  the space $L_{2}(a, b)$.  It is easy to see, that the symmetric
operator $L_{\Gamma, \beta}$  is a finite rank perturbation of the
 self--adjoint operator $L$ in the space
 $L_{2}(R^1)$ and corresponds to self--adjoint boundary conditions
 $\psi(a-0)=\psi(a+0),$\, $\psi'(a-0)=\psi'(a+0),$\,$\psi(b-0)=\psi(b+0),
$\,$\psi'(b-0)=\psi'(b+0).$
 Therefore,  [20] the operator $L_{\Gamma, \beta}$ is self--adjoint
 in the space $L_{2}(R^1).$
Since  the operator $[L_{\Gamma,\beta}^{(a,b)}]^{-1}$ is compact
and the operators $L_D$ and $L_N$ have absolutely continuous
spectrum $[0, +\infty)$ and the spectrum of the operator
$L_{\Gamma,\beta}^{(a,b)}$ is discrete with only possible limit
point $\lambda=\infty $ then the spectrum of the operator
$L=L_{D}\oplus L_{\Gamma, \beta}^{(a,b)}\oplus L_{N}$ and
consequently the  spectrum of  the operator $L_{\Gamma,\beta}$ can be
only discrete on the negative half-axis since the self--adjoint
operator $L_{\Gamma,\beta}$ is a finite rank perturbation of the
operator $L$.

On the other hand, the
  domain~$ D(L_{\Gamma, \beta})$ possesses the properties required in
  Definition~\ref{D:3}. Indeed, it follows from
  representation~(\ref{eq:**}) that if a function~$\psi(x)$ has
  boundary values on~$\Gamma$, then the same is true for the
  function~$\chi(x)\cdot\psi(x)$. The boundary data for the
  function~$\chi(x)\cdot\psi(x)$ coincide with the boundary data for
  the function~$\psi(x)$ multiplied by the function~$\chi$, that is,
  $\frac{d (\chi\psi)}{d \mu}= \chi \frac{d \psi}{d \mu}$, which means
  that~$(\chi\psi)_{r}=\chi\psi_{r}$, etc. This shows that, if~$\psi
  \in D(L_{\Gamma, \beta})$, then~$\chi\psi \in D(L_{\Gamma, \beta})$.

  Hence, the self-adjoint operator~$L_{\Gamma, \beta}$ describes a
  local interaction on~$\Gamma$ according to
  Definition~\ref{D:1}. Since the function~$\chi$ has trivial boundary
  data and~$\frac{d \chi}{d \mu}= 0$, $\frac{d \chi'}{d \mu}= 0$,
  $\chi'_{r}=0$, $\chi_{r} =\chi$, it follows that this function
  satisfies boundary conditions~(\ref{eq:**}) and, consequently, $\chi
  \in D(L_{\Gamma, \beta})$. It now follows from Definition~\ref{D:3}
  that the self-adjoint operator~$L_{\Gamma, \beta}$ defines a
  $\delta'$-interaction on the set~$\Gamma$.
\end{proof}

In order to extend the results of Theorem~\ref{Th:1} to a general
case, we will need the following definition.

\begin{definition}\label{D:5}
  We say that a real-valued function~$\beta$ defined on a set~$\Gamma$
  with a measure~$\mu$ assumes negative values on an infinite number
  of subsets of~$\Gamma$ if for any natural~$N$ there
  exists~$\varepsilon > 0$ and a collection of closed measurable
  nonintersecting subsets~$\Gamma_{k} \subset\Gamma$, $\mu
  (\Gamma_{k})> 0$, $k=1,...,N$, such that the function~$\beta(x)$
  assumes strictly negative values on~$\Gamma_{k} $, $\beta(x)\leq
  -\varepsilon$, $ x\in \Gamma_{k}$, $k=1,...,N$.
\end{definition}

\begin{theorem}\label{Th:3}
  Let a real-valued function~$\beta$, defined on a closed bounded
  set~$\Gamma$ of Lebesgue measure zero, be absolutely integrable with
  respect to a Radon measure~$\mu$ and assume negative values on an
  infinite number of subsets of~$\Gamma$. Then the Schr\"odinger
  operator~$L_{\Gamma, \beta}$ with $\delta'$-interaction on~$\Gamma$,
  having intensity~$\beta$, is a self-adjoint operator on the
  space~$L_{2}(R^1)$ and has an infinite number of negative eigenvalues,~$\lambda_{n}\rightarrow -\infty$.
\end{theorem}

\begin{proof}
  The proof is similar to the proof of Theorem~\ref{Th:1}.
Let the
  conditions of the theorem be satisfied. Then, by Theorem~2, the
  operator~$L_{\Gamma, \beta}$ is self-adjoint on~$L_{2}(R^1)$
  and the negative spectrum of
  the operator~$L_{\Gamma, \beta}$ is discrete.
  Let us
  show that the operator~$L_{\Gamma, \beta}$ has an infinite number of
  negative eigenvalues. To this end, it is sufficient to show that
  there exists an

$N$--dimensional subspace ${\cal L}_N$
of the domain of $L_{\gamma,\beta}$ such that
  $(L_{\Gamma,\beta} u,u )<0, $  for any $u \in
\mathcal{L}_N,$\,\,$u\neq 0$
  for any natural~$N$. Fix~$N$. By the conditions of the theorem
there
  exist~$N$ nonintersecting closed subsets~$\Gamma_{k} \subset\Gamma$,
  $\mu (\Gamma_{k})> 0$, and~$\varepsilon> 0$ such that~$\beta(x)\leq
  -\varepsilon$ for~$x\in \Gamma_{k}$, $k=1,...,N$. Consider analogues
  of the test functions of Section~\ref{S:4}. Since the number of
  subsets~$\Gamma_{k}$ is finite, they are closed and nonintersecting,
  there is~$\delta>0$ such that all $\delta$-neighborhoods
  $\mathcal{U}_{\delta}(\Gamma_{k})=\{y: |y-x|< \delta, \,x\in
  \Gamma_{k}\}$ of the sets~$\Gamma_{k}$ are also pairwise
  nonintersecting. Let us construct a test function for each
  set~$\Gamma_{k}$ as follows. Consider the function~$\chi_{k}(x)\in
  C_0^{\infty}(R^1)$ that equals~$1$ on~$\Gamma_{k}$, takes values
  between $0$ and $1$, and equals to zero outside
  of~$\mathcal{U}_{\delta}(\Gamma_{k})$. Such a step function can be
  constructed as usual by making a smooth function from the
  characteristic function of the
  set~$\mathcal{U}_{\frac{\delta}{2}}(\Gamma_{k})$. As a candidate for
  the test function, we take
  \begin{equation}\label{eq:5*}
    \hat{t}_k(x;\beta,\Gamma_{k},\delta)=\int\limits_{a}^{x}\,\chi_{k}(\xi)\,d\xi+\int\limits_{(a,x)}\,\beta(\xi)\chi_{k}(\xi)\,d\mu(\xi),
  \end{equation}
  where the number~$a$ is chosen so that all bounded
  sets~$\mathcal{U}_{\delta}(\Gamma_{k})$, $k=1,...,N$, would lie to
  the right of the point~$a$. For~$x$ that lie on the right of the
  set~$\Gamma$, this function takes the constant value~$c_{k}$. While
  the function~$t_k$ does not belong to the space~$L_{2}(R^1)$, we can
  turn it into a function with compact support using two parabolas on
  the interval~$[l, l+2r]$ that lies to the right of~$\Gamma$. We thus
  get the test function
  \begin{equation}\label{eq:35}
    t_{k}(x;\beta,\Gamma_{k},\delta,l,r)=\left\{
      \begin{array}{ll}
        \hat{t},  &  x \leq l,\\[2mm]
        -\frac{c_{k}}{2r^{2}}(l-x)^2 + c_{k}, & l \leq x \leq l+r, \\[2mm]
        \frac{c_{k}}{2r^{2}}(l+2r-x)^2,& l +r  \leq x \leq l+2r , \\[2mm]
        0,  & l+2r < x.
      \end{array}
    \right.
  \end{equation}
  Here, the parameters~$l$ and~$r$ may depend on~$k$.
\begin{proposition}
  The main properties of the test functions~$t_k$~(\ref{eq:35}) are the following:
  \begin{itemize}
  \item [$1^{0}$] $t_k \in \mathcal{D}(L_{\Gamma, \beta})$ if~$\Gamma
    \subset ( -\infty, l)$.
  \item [$2^{0}$] By choosing~$\delta$ sufficiently small and~$r$
    sufficiently large, we have
    \begin{equation}\label{eq:7*}
      (L_{\Gamma, \beta} t_k, t_k)\leq
      -\frac{1}{8}\varepsilon\mu(\Gamma_{k}),
    \end{equation}
    that is, the quadratic form takes negative values.
  \item [$3^{0}$] The quadratic form of the linear
    combination~$t=\sum\limits_{k=1}^{N}\,a_k\cdot t_k$ of test
    functions that satisfy the condition~$1^{0}$, if~$l_k$ and~$r_k$
    are chosen so that the intervals~$[l_k, l_{k}+2r_k]$ are pairwise
    disjoint, takes negative values,
    \begin{equation}\label{eq:7**}
      (L_{\Gamma, \beta} t, t)= \sum\limits_{k=1}^{N}\,|a_k|^2
      (L_{\Gamma, \beta} t_k, t_k) \leq -\frac{1}{8}\varepsilon
      \min\limits_k \mu(\Gamma_{k}) \sum\limits_{k=1}^{N}\,|a_k|^2 <0.
    \end{equation}
  \end{itemize}
  \end{proposition}
  If these three conditions are satisfied, then the proof is
  finished by applying the variational minimax principle~[26] as
  in the proof of Theorem~1.

  Let us now prove that test functions satisfy
  properties~$1^{0}$---$3^{0}$. The first property is clearly
  satisfied by the construction of~$t_k$ and~$\hat{t}_k$
  in~(\ref{eq:5*}) and~(\ref{eq:35}) and the definition of the
  operator~$L_{\Gamma, \beta}$. The second property is most
  important. Since the function~$\beta$ is absolutely integrable
  on~$\Gamma$ with respect to the Radon measure~$\mu$ and~$ 0\leq
  \chi_k \leq 1$, we see that there exists small~$\delta$ such that
  \begin{equation}\label{eq:8*}
    \Bigr|\int\limits_{\mathcal{U}_{\delta}(\Gamma_k)\cap\Gamma}\,\beta(\xi)\chi_{k}(\xi)\,d\mu(\xi)-
    \int\limits_{\Gamma_k}\,\beta(\xi)\,d\mu(\xi)\Bigl | <
    \frac{1}{2}\varepsilon \mu(\Gamma_{k}).
  \end{equation}
  Moreover, since the set~$\Gamma$ has Lebesgue measure zero, there
  exists a small~$\delta$ such that the following estimate holds for
  the Lebesgue measure of the
  set~$\mathcal{U}_{\delta}(\Gamma_{k})$:
  \begin{equation}\label{eq:9*}
    |\mathcal{U}_{\delta}(\Gamma_k)| \leq \frac{1}{4} \varepsilon
    \mu(\Gamma_{k}).
  \end{equation}

  If inequalities~(\ref{eq:8*}) and~(\ref{eq:9*}) hold, then the
  constant~$c_{k}$, which is equal to the value of the
  function~$\hat{t}$ for large~$k$, satisfies the estimate
  \begin{equation}\label{eq:10*}
    |c_k| \leq (\frac{3}{4}\varepsilon + ||\beta||_{L_1(\Gamma,
      d\mu)})  \mu(\Gamma_{k}).
  \end{equation}
  By choosing~$r_{k}$ large enough, we have
  \begin{equation}\label{eq:11*}
    \int\limits_{l_k}^{l_{k}+2r_k}\,|t'(x)|^2\,dx \leq
    \frac{1}{8}\varepsilon \mu(\Gamma_{k}).
  \end{equation}

  In virtue of Green's first formula~(\ref{eq:31a}), since the
  function~$t_k$ satisfies the boundary conditions~(\ref{eq:33}) and
  because~$t'_k(x)=\chi_k(x)$ for~$x\leq l_k$, we have
  \begin{equation}\label{eq:42}
    (L_{\Gamma, \beta} t_k, t_k)= \int\limits_{a}^{l_{k}}\,|\chi_k
    (x)|^2\,dx+
    \int\limits_{l_k}^{l_{k}+2r_k}\,|t'_k(x)|^2\,dx+\int\limits_{\Gamma}\,\beta(\xi)|\chi_{k}(\xi)|^2\,d\mu(\xi).
  \end{equation}
  The first integral~$\mathcal{I}_1$ in~(\ref{eq:42}) can be estimated
  in terms of the Lebesgue measure~$\mathcal{U}_{\delta}(\Gamma_{k})$,
  since values of the function~$\chi_k(x)$ belong to the
  interval~$[0,1]$. The second
  integral~$\mathcal{I}_2=\frac{2}{3}c_k^2 \cdot r_k^{-1}$ can be
  explicitly calculated, since the function~$t_k'(x)$ on the
  interval~$[l_k,l_{k}+2r_k]$ consists of two parabolas
  by~(\ref{eq:35}). The third integral~$\mathcal{I}_3$
  in~(\ref{eq:42}) can be estimated as follows:
  $$
  \mathcal{I}_3=\int\limits_{\Gamma_k}\,\beta(\xi)\,d\xi +
  \Bigr|\int\limits_{\mathcal{U}_{\delta}
    (\Gamma_k)}\,\beta(\xi)\chi_{k}^2(\xi)\,d\mu(\xi)-
  \int\limits_{\Gamma_k}\,\beta(\xi)\,d\mu(\xi)\Bigl |.
  $$
  Since, by choosing sufficiently small~$\delta$ and sufficiently
  large~$r_k$ we can satisfy estimates~(\ref{eq:8*})--(\ref{eq:11*}),
  we see that the quadratic form~$(L_{\Gamma, \beta} t_k, t_k)$ is
  negative, i.e., inequality~(\ref{eq:7*}) is satisfied.

  Consider now property~$3^{0}$. Since the
  intervals~$(l_k,l_{k}+2r_k)$ and the
  regions~$\mathcal{U}_{\delta}(\Gamma_{k})$ are mutually disjoint, we
  have that~$(L_{\Gamma, \beta} t_k, t_j)=0$ for~$k\neq j$. This leads
  to property~(\ref{eq:7**}).
\end{proof}

For nonlocal interactions we may get a behaviour different from
the one in the local case. We illustrate this fact by the
following example.

\begin{example}
It is not possible that the same function is eigenfunction
with negative eigenvalue of two different Schr\"{o}dinger
operators with local $\delta$ and $\delta'$ interactions. This
is not true for nonlocal point interactions.

Indeed,  let $A_1$ be the self--adjoint operator in $L_2(R^1)$ that
corresponds to the two--point nonlocal interaction in the points $x_1=-1$
and $x_2=1$ described by following self--adjoint  boundary
conditions
\begin{equation}
\begin{array}{l}\label{neq:1}
  \psi'(x_j+0)-\psi'(x_j-0)=0,\\[2mm]
  \psi'(x_j+0)+\psi'(x_j-0)+\psi'(x_1+0)-\psi'(x_1-0)+\psi(x_2+0)-\psi(x_2-0)=0,\,\,j=1,2.
\end{array}
\end{equation}

There exists a unique negative eigenvalue $-\lambda_0^2$ of the operator
$A_1$ where the number $\lambda_0$  is the positive root of the
characteristic equation $\lambda_0=1+\tanh \lambda_0$,\,\,
$\lambda_0\approx 1.968$. The eigenfunction $\psi_0(x)$ is odd
$\psi_0(-x)=-\psi_0(x)$  and has the form:
\begin{equation}
  \label{neq:2}
 \psi_0(x) = \left\{
 \begin{array}{ll}
      - \frac{\sinh \lambda_0 x }{\cosh \lambda_0},&  0 \le x <
      1,\\[2mm]
e^{-\lambda_0 (x-1)},& 1 < x < +\infty
 \end{array}
      \right.
\end{equation}
Let us consider the self--adjoint operator $A_2$ that corresponds to
the local $\delta'$ interaction in the points $x_1=-1$ and $x_2=1$  with
intensity $\beta=-1$.  The domain of the operator $A_2$ is given by
self--adjoint conditions
\begin{equation}
\begin{array}{l}\label{neq:3}
  \psi'(x_j+0)-\psi'(x_j-0)=0,\\[2mm]
  \psi(x_j+0)-\psi(x_j-0)=-\psi'(x_j),\,\,j=1,2.
\end{array}
\end{equation}
It is easy to check that the function $\psi_0(x)$  of (\ref{neq:2}) satisfies
the boundary conditions (\ref{neq:3}), i.e. is an
eigenfunction of the operator $A_2$.
However, there exists one more even eigenfunction
$\psi_1(x)=\psi_1(-x)$ of the form
\begin{equation}
  \label{neq:4}
 \psi_1(x) = \left\{
 \begin{array}{ll}
      - \frac{\cosh \lambda_1 x }{\sinh \lambda_1},&  0 < x <
      1,\\[2mm]
e^{-\lambda_1 (x-1)},& 1 < x < +\infty
 \end{array}
      \right.
\end{equation}
with negative eigenvalue  $-\lambda_1^2$ where $ \lambda_1$ is the
positive root of the equation $ \lambda_1=1+\coth  \lambda_1,$\,\, $
\lambda_1\approx 2.03$.

\end{example}

\section{Deficiency subspaces}

In this section we give for arbitrary closed subsets $\Gamma$ of
$\R$ with Lebesgue measure zero the deficiency subspaces of the
operator $L_{min,\Gamma}$. This result can be used for the
construction of Hamiltonians describing an interaction which takes
place inside $\Gamma$.

First we fix some notation
and consider any
symmetric operator $S$ in any complex  Hilbert space ${\cal H}$ such that
the deficiency subspaces $\ran(S\pm i)^{\perp}$ of $S$ have the same
Hilbert space
dimension. For every unitary transformation
$U:\ran(S+i)^{\perp} \longrightarrow \ran(S-i)^{\perp}$
put
\begin{eqnarray}\label{neu1}
D(S_U) &  := & \{ f+Uf +h: f\in \ran (S+i)^{\perp}, h\in D(\bar{S}) \},\nonumber \\
S_U & := & S^* \lceil D(S_U).
\end{eqnarray}
By von Neumann's first and second formula,
 the mapping $U\mapsto S_U$ from the set of
unitary transformations $U:\ran(S+i)^{\perp} \longrightarrow
\ran(S-i)^{\perp}$ onto the set of self--adjoint extensions of $S$
is bijective. Moreover every $f\in D(S^*)$ can be uniquely
represented as
$$ f= f_+ + f_- +h, \quad f_{\pm} \in \ran (S \pm i)^{\perp},\,
h\in D(\bar{S}).$$
Thus
\begin{eqnarray}\label{neu1b}
f_-= Uf_+, \mbox{ if $f_{\pm}\in \ran(S\pm i)^{\perp}, h\in D(\bar{S})$ and
$f_+ + f_- +h\in D(S_U)$.}
\end{eqnarray}

Let $D$ be a closed linear subspace of $\ran(S+i)^{\perp}$. Put
\begin{eqnarray}\label{neu2}
D(S_U^D) &  := & \{ f+Uf +h: f\in
\ran (S+i)^{\perp}\cap  D^{\perp}, h\in D(\bar{S}) \},\nonumber \\
S_U^D & := & S^* \lceil D(S_U^D).
\end{eqnarray}
$f\in \ran(S_U^D +i)^{\perp}$ if and only if $f\in \ran(S+i)^{\perp}$ and
$$ f\perp (S^*+i)\,(f_+ + Uf_+) = 2i f_+, \quad f_+ \in \ran(S+i)^{\perp} \cap D^{\perp}.$$
Thus $\ran(S_U^D+i)^{\perp} = D$ and
 we have proved the following lemma:

\begin{lemma}\label{tneu0}
Let $V:\ran(S+i)^{\perp} \longrightarrow \ran(S-i)^{\perp}$
be any linear mapping such that
$Vf=Uf$ for all $f\in \ran(S+i)^{\perp} \cap D^{\perp}$ and $V\lceil D$ is a
unitary mapping from $D$ onto $\{Uf: f\in D \}$. Then $V$ is a unitary mapping
from $\ran(S+i)^{\perp}$ onto $\ran(S-i)^{\perp}$, $S_U^D= S_V^D$ is a
restriction of $S_U$ and $S_V$ and
\begin{eqnarray}\label{neu3}
\ran(S_V^D+i)^{\perp} = D
\end{eqnarray}
\end{lemma}

By Krein's formula and (\ref{neu3}),
$(S_U+i)^{-1} - (S_V+i)^{-1}$ is a finite rank
operator with rank $\dim D$, provided $D$ is finite dimensional.
By Weyl's essential
spectrum theorem, the Birman-Kuroda theorem, and a theorem by Krein
this implies the following result:

\begin{lemma}\label{tneu2}
Let $S$ be a symmetric operator in the Hilbert space ${\cal H}$,
$D$ a finite dimensional subspace of $\ran(S+i)^{\perp}$ and $U$
and $V$ unitary transformations from $\ran(S+i)^{\perp}$ onto
$\ran(S-i)^{\perp}$ which coincide on $D^{\perp}\cap \ran
(S+i)^{\perp}$. Then the self--adjoint extensions $S_U$ and $S_V$
(cf. (\ref{neu1})) have the same essential and the same absolutely
continuous spectrum. and the number, counting multiplicities, of
eigenvalues of $S_V$ below the minimum of the essential spectrum
of $S_V$ is less than or equal to $\dim D$.
\end{lemma}

Now let us consider explicite examples.
Let $\Gamma$ be a closed subset of $R^d$ with Lebesgue measure zero and
$2\alpha \in \N$. Let $S$ be the symmetric operator in $L^2(R^d)$ defined
as follows:
\begin{eqnarray*}
D(S) & := & C_0^{\infty} (R^d\setminus \Gamma),\\
Sf & : = & (- \Delta)^{\alpha} f, \quad  f\in D(S).
\end{eqnarray*}
Based on ideas in \cite{K82} and with the aid of the theorem on the
spectral synthesis in Sobolev spaces \cite{H81} one has determined the
deficiency subspaces
\[ \ran (S-z)^{\perp}, \quad z\in \C \setminus [0, \infty), \]
of the operator $S$, cf. \cite{Br95}, Example 2.8.
In order to formulate
this result in the case we are interested in, i.e. $d=1= \alpha$, we use
the following notation:
\begin{eqnarray}\label{def1}
 g_z(x) := \frac{i}{2\sqrt{z}} e^{i\sqrt{z}\vert x \vert}, \quad x\in \R, \,
z\in \C\setminus [0, \infty),
\end{eqnarray}
where the square root has to be chosen such that the imaginary part of
$\sqrt{z}$ is positive,
${\cal M}_{\Gamma}$ denotes the set of positive Radon measures
 on
$\R$ with compact support in $\Gamma$
and
\begin{eqnarray}\label{def3}
{\cal T}_{z,\Gamma} := \{ g_z*\mu: \mu \in {\cal M}_{\Gamma} \}
\cup \{ (g_z*\nu)': \nu \in {\cal M}_{\Gamma} \}.
\end{eqnarray}
Since every finite positive Radon measure $\mu$ on $\R$ belongs to
the Sobolev space $W_2^{-1}(\R)$, we get the following result:

\begin{lemma}\label{tdef1} {\em (cf. \cite{Br95}, Example 2.8)}
Let $\Gamma$ be a closed subset of $\R$ with Lebesgue measure
zero. Then ${\cal T}_{z,\Gamma}$, defined by (\ref{def3}), is a
total subset of the deficiency subspaces ${\mathfrak N}_{z,
\Gamma}\equiv \ran (L_{\min,\Gamma}-\overline{z})^{\bot}\equiv
\ker (L_{\max,\Gamma}-z)$ , i.e. ${\mathfrak N}_{z, \Gamma}$ is
the closure of the span of ${\cal T}_{z,\Gamma}$.
\end{lemma}

\noindent
Let $(\mu_n)$ be a sequence of finite positive Radon measures converging
weakly to the finite positive Radon measure $\mu$. Then, by the dominated
convergence theorem, the sequences of the Fourier
transforms of $(g_z*\mu_n)$ and $((g_z*\mu_n)')$
converge in $L^2(\R)$ to the Fourier transform of $g_z*\mu$
and $(g_z*\mu)'$, respectively. Hence the sequences $(g_z*\mu_n)$ and
$((g_z*\mu_n)')$ converge in $L^2(\R)$ to $g_z*\mu$ and $(g_z*\mu)'$,
respectively. Moreover for every finite positive Radon measure on $\R$ there
exist positive Radon measures $\mu_n$, $n\in \N$,  such that the support
of $\mu_n$ is a finite subset of the support of $\mu$ for every $n\in \N$
and the sequence $(\mu_n)$ converges weakly to $\mu$. By Lemma \ref{tdef1},
this implies that
\begin{eqnarray*}
\{ g_z(x-\gamma): \gamma \in \Sigma \} \cup \{ g_z'(x-\gamma):
\gamma \in \Sigma \}
\end{eqnarray*}
is a total subset of the deficiency subspace ${\mathfrak N}_{z,
\Gamma}$, if $\Sigma$ is dense in $\Gamma$.
Moreover if $\gamma$ is not an isolated point of $\Gamma$, then
$g_z(x-\gamma)$ and $g_z'(x-\gamma)$ belong to the closure of the
span of the set $\{ g_z(x-\gamma): \gamma\in \Sigma \} $. Thus we
get the following result:

\begin{proposition}\label{tdef2}
Let $\Gamma$ be a closed subset of $\R$ with Lebesgue measure
zero. Then for every $z\in \C \setminus [0, \infty)$
\begin{eqnarray}\label{def5}
B(\Sigma_1, \Sigma_2)=\{ g_z(x-\gamma): \gamma\in \Sigma_1 \} \cup
\{ g_z'(x-\gamma): \gamma \in \Sigma_2 \}
\end{eqnarray}
is a total subset of the deficiency subspace ${\mathfrak N}_{z,
\Gamma}$, if, and only if, $\Sigma_1$ is dense in $\Gamma$ and
$\Sigma_2$ contains the set of all isolated points of $\Gamma$.
\end{proposition}
\begin{proof}

Let us show that the conditions on $\Sigma_1$ and $\Sigma_2$ are
necessary for totality of the set $B(\Sigma_1, \Sigma_2)$ in
${\mathfrak N}_{z, \Gamma}$. Let $\Sigma_1$ be not dense in
$\Gamma$. Then, there exists a partition  $\Gamma=\Gamma_1\cup
\Gamma_2$ on two not empty, not intersecting closed subsets
$\Gamma_1$ and $\Gamma_2$ and $\Sigma_1 \subset \Gamma_1$. In this
case,
$$
{\mathfrak N}_{z, \Gamma}={\mathfrak N}_{z, \Gamma_1}\dot{+}
{\mathfrak N}_{z, \Gamma_2},
$$
where the sum  is direct and corresponding to this sum the skew
projectors $P_j:$\,${\mathfrak N}_{z, \Gamma} \rightarrow
{\mathfrak N}_{z, \Gamma_j},$\,\,$j=1,2$ are bounded operators.
Let us now show that in this case the set $B(\Gamma_1,\Gamma)$
which is larger than $B(\Sigma_1, \Sigma_2)$ will not be total in
the deficiency  space ${\mathfrak N}_{z, \Gamma}.$  If
$B(\Gamma_1,\Gamma)$ would be total in ${\mathfrak N}_{z, \Gamma}$
then the set $B(\Gamma_2)=\{ g_z'(x-\gamma): \gamma \in \Gamma_2
\}$
 would be total
in ${\mathfrak N}_{z, \Gamma_2}$. However, it is not possible. In
fact, a linear continuous with respect to the metric of $L_2(R^1)$ 
functional $e(f)=\int\limits_{R^1}\, f(x)\,dx$ is defined on the
whole ${\mathfrak N}_{z, \Gamma_2}$ and is equal to zero on
${\mbox{span}}B(\Gamma_2)$ but it is equal to $\displaystyle-
\frac{1}{z}$ on a function $g_z(x-\gamma) \in{\mathfrak N}_{z,
\Gamma_2}$. One can prove that if $\Sigma_2$ does not contain all
isolated points $x_0$ in $\Gamma$ then even $B(\Gamma,
\Gamma\setminus \{x_0\})$ can not be total in ${\mathfrak N}_{z,
\Gamma}$. In this case, the deficiency space ${\mathfrak N}_{z,
\{x_0\}}$ is two-dimensional and a skew projection ${\mbox{span}}
B(\Gamma, \Gamma\setminus \{x_0\})$ on ${\mathfrak N}_{z,
\{x_0\}}$ is a one-dimensional subspace.
\end{proof}


\begin{example}

Let $\Gamma$ be a closed subset of $\R$ with Lebesgue measure zero
and put $S:= L_{min,\Gamma}$. As pointed out in lemma \ref{tneu2}
one may get far reaching results on the spectral properties of one
self--adjoint extension $S_V$ of $S$ with the aid of another
self--adjoint extension $S_U$ of $S$. Since one knows the spectral
properties of the free quantum mechanical Hamiltonian, it is
interesting to determine the unitary mapping $U$ such that $S_U$
is the free quantum mechanical Hamiltonian, i.e.
\begin{eqnarray}\label{neu4}
S_U\psi(x) = - \psi''(x), \quad \psi\in D(S_U) = W_2^{2}(\R).
\end{eqnarray}
Passing to Fourier transforms one sees that $g_{-i}*\mu - g_{i}
*\mu\in W_2^{2}(\R)$ and $(g_{-i}*\mu)' - (g_{+i}*\mu)'\in
W_2^{2}(\R)$ for every $\mu \in {\cal M}_{\Gamma}$.
 By (\ref{neu1b}), this implies
that
\begin{eqnarray}\label{neu5}
Ug_{-i}*\mu = - g_i *\mu\mbox{ and } U((g_{-i}*\mu)') = -(g_i * \mu)',
\quad \mu \in {\cal M}_{\Gamma}.
\end{eqnarray}
Now fix $\mu\in {\cal M}_{\Gamma}$  and $\alpha \in \SSS^1$. By
Lemma \ref{tneu0} there exists a unique self--adjoint extension
$A$ of $S$ such that
\begin{eqnarray}\label{def6}
\psi_0 := (g_{-i}*\mu)' + \alpha (g_i * \mu)' \in D(A)
\end{eqnarray}
and
$A$ and $S_U$ have a common restriction $T$ such that
$\ran(T+i)^{\perp}$ is spanned by $(g_{-i} * \mu)'$.
By Lemma \ref{tneu2},
$A$ and $S_U$ have the same essential spectrum
and the same absolutely continuous spectrum and hence
\begin{eqnarray}\label{def6b}
\sigma_{ess}(A) = [0, \infty) = \sigma_{ac}(A),
\end{eqnarray}
and the number, counting multiplicities, of negative eigenvalues of $A$
is less than or equal to one.
\end{example}

{\it Acknowledgments.\/}   The second author (L.N.) expresses his
gratitude to DFG for a financial support of the project DFG  BR
1686/2-1  and thanks the Institute of Mathematics at TU  of
Clausthal  for the warm hospitality.

\end{document}